\documentclass[reqno,centertags, 12pt]{amsart}
\usepackage{amsmath,amsthm,amscd,amssymb}
\usepackage{latexsym}
\usepackage{graphicx}
\usepackage[mathscr]{eucal}
\newcommand{\bbC}{{\mathbb{C}}}
\newcommand{\bbD}{{\mathbb{D}}}
\newcommand{\bbE}{{\mathbb{E}}}

\newcommand{\bbP}{{\mathbb{P}}}

\newcommand{\bbR}{{\mathbb{R}}}

\newcommand{\bbU}{{\mathbb{U}}}
\newcommand{\bbZ}{{\mathbb{Z}}}

\newcommand{\fre}{{\frak{e}}}

\newcommand{\calC}{{\mathcal{C}}}

\newcommand{\calG}{{\mathcal G}}
\newcommand{\calH}{{\mathcal H}}

\newcommand{\calL}{{\mathcal L}}
\newcommand{\calM}{{\mathcal M}}

\newcommand{\bdone}{{\boldsymbol{1}}}

\newcommand{\no}{\nonumber}
\newcommand{\lb}{\label}

\newcommand{\ti}{\tilde  }
\newcommand{\wti}{\widetilde  }

\newcommand{\tr}{\text{\rm{Tr}}}

\newcommand{\supp}{\text{\rm{supp}}}

\newcommand{\bi}{\bibitem}

\newcommand{\beq}{\begin{equation}}
\newcommand{\eeq}{\end{equation}}
\newcommand{\ba}{\begin{align}}
\newcommand{\ea}{\end{align}}

\newcounter{smalllist}
\newenvironment{SL}{\begin{list}{{\rm\roman{smalllist})}}{%
\setlength{\topsep}{0mm}\setlength{\parsep}{0mm}\setlength{\itemsep}{0mm}%
\setlength{\labelwidth}{2em}\setlength{\leftmargin}{2em}\usecounter{smalllist}%
}}{\end{list}}

\newcommand{\comm}[1]{}

\DeclareMathOperator{\re}{Re}
\DeclareMathOperator{\im}{Im}

\allowdisplaybreaks
\numberwithin{equation}{section}

\newtheorem{theorem}{Theorem}[section]

\newtheorem*{p2.1}{Proposition 2.1}
\newtheorem{proposition}[theorem]{Proposition}
\newtheorem{lemma}[theorem]{Lemma}

\theoremstyle{definition}
\newtheorem{example}[theorem]{Example}

\newtheorem*{remark}{Remark}
\newtheorem*{remarks}{Remarks}

\newcommand{\jap}[1]{\langle #1 \rangle}
\newcommand{\norm}[1]{\lVert#1\rVert}
\newcommand{\overbar}{\overline}

\begin{document}

\title[Large Deviations and Sum Rules]{Large Deviations and Sum Rules for Spectral Theory \\ ~ \\ \small{A Pedagogical Approach}}

\author[J.~Breuer, B.~Simon, and O.~Zeitouni]{Jonathan~Breuer$^{1,4}$, Barry~Simon$^{2,5}$,\\and
Ofer Zeitouni$^{3,6}$}

\thanks{$^1$ Institute of Mathematics, The Hebrew University of Jerusalem, Jerusalem, 91904, Israel.
E-mail: jbreuer@math.huji.ac.il}

\thanks{$^2$ Departments of Mathematics and Physics, Mathematics 253-37, California Institute of Technology, Pasadena, CA 91125.
E-mail: bsimon@caltech.edu}

\thanks{$^3$ Faculty of Mathematics, Weizmann Institute of Science, POB 26, Rehovot 76100, Israel and Courant Institute, NYU.
E-mail: ofer.zeitouni@weizmann.ac.il.}

\thanks{$^4$ Research supported in part by Israeli BSF Grant No. 2014337.}

\thanks{$^5$  Research supported in part by NSF grant DMS-1265592 and in part by Israeli BSF Grant No. 2014337.}

\thanks {$^6$ Research supported in part by a grant from the Israel Science Foundation.}

\

\date{\today}

\begin{abstract} This is a pedagogical exposition of the large deviation approach to sum rules pioneered by Gamboa, Nagel and Rouault.  We'll explain how to use their ideas to recover the Szeg\H{o} and Killip--Simon Theorems.  The primary audience is spectral theorists and people working on orthogonal polynomials who have limited familiarity with the theory of large deviations.
\end{abstract}

\maketitle

\section{Introduction and Sum Rules} \lb{s1}

This note discusses a new approach to sum rules in the spectral theory of orthogonal polynomials on the unit circle (OPUC) and real line (OPRL).  We are given a probability measure, $d\mu$, on $\partial\bbD$ or $\bbR$ of the form:
\begin{equation}\label{1.1}
  d\mu(\theta) = w(\theta)\frac{d\theta}{2\pi}+d\mu_s \, (OPUC); \ d\mu(x) = w(x) dx +d\mu_s \, (OPRL)
\end{equation}
where $d\mu_s$ is singular w.r.t $d\theta$ or $dx$.  The recursion relations for OPUC are given for the monic OPs, $\{\Phi_n\}_{n=0}^\infty$, by
\begin{equation}\label{1.2}
  \Phi_{n+1}(z)=z\Phi_n(z) - \overbar{\alpha}_n \Phi^*_n(z); \quad \Phi_0 \equiv \bdone; \qquad \Phi^*_n(z) = z^n \overbar{\Phi_n\left(\frac{1}{\bar{z}}\right)}
\end{equation}
with Verblunsky coefficients $|\alpha_n| \leq 1$,
and for OPRL and the orthonormal polynomials, $\{p_n\}_{n=0}^\infty$, by
\begin{equation}\label{1.3}
  xp_n(x) = a_{n+1}p_{n+1}(x) + b_{n+1}p_n(x) + a_n p_{n-1}(x); \qquad p_{-1} \equiv 0
\end{equation}
with Jacobi parameters $b_n \in \mathbb{R}$, $a_n \geq 0$.

Verblunsky's Theorem (see \cite{OPUC1, SzThm}) says that there is a bijection $V: \prod_{n=0}^\infty \bbD \to \{\mu \textrm{ on } \partial\bbD \, | \, \mu(\partial\bbD) = 1; \textrm{ support of }\mu \textrm{ is infinite}\}$ that maps the Verblunsky coefficients, $\{\alpha_n\}_{n=0}^\infty$, to a measure with those Verblunsky coefficients.   Here $\bbD$ will denote the open unit disc in the complex plane, $\bbC$, and we say the support is infinite if it is not a finite set of points.  Similarly, for each $n$, there is a bijection, $V_n$, of $\left(\prod_{j=0}^{n-2} \bbD\right) \times \partial\bbD$ to probability measures on $\partial\bbD$ with exactly $n$ points in their support (i.e. $\alpha_j \in \bbD,  j=0,\dots,n-2; \, \alpha_{n-1} \in \partial\bbD$.)  Moreover, the maps are homeomorphisms if the $\alpha$'s are given the topology of pointwise convergence (product topology) and the measures the topology of vague convergence (i.e.\ weak topology as functionals on all continuous functions).

In some sense, it is natural to consider all measures at once and the corresponding set of possible Verblunsky coefficient sequences.  On this larger set, the topology is no longer a product topology (which it can't be since it is a union of products but not a product itself).
A Verblunsky coefficient sequence is a sequence $\alpha=\{\alpha_n\}_{n=0}^{N(\alpha)-1}$ with $N(\alpha) \leq \infty$.  If $N(\alpha)<\infty$ the sequence has $N(\alpha)$ elements with the last one in $\partial\bbD$ and the others in $\bbD$. The topology is metrizable and $\alpha^{(j)}$ converges to $\alpha^{(\infty)}$ by the following rule: If $N(\alpha^{(\infty)}) = \infty$, then we require that $N(\alpha^{(j)}) \to \infty$ and $\alpha^{(j)}_k \rightarrow \alpha_k^{(\infty)}$ for each fixed $k$ as $j \to \infty$.  If $N(\alpha^{(\infty)}) = N_0 < \infty$, we require that $\alpha^{(j)}_k \to \alpha^{(\infty)}_k$ for $k = 0,\dots,N_0-1$.  In this topology the space of Verblunsky coefficients sequences is compact and $V$ maps this space homeomorphically onto the space of all probability measures with the weak topology. This shows also that the topology on Verblunsky coefficients is metrizable.

Similarly, Favard's Theorem (see \cite{SzThm}) says that there is a bijection, $J: \prod_{n=1}^{\infty} \left(\bbR,(0,\infty)\right)$ with $\sup(|a_n|+|b_n|) < \infty \to \{\mu \textrm{ on } \bbR \,| \, \mu(\bbR) = 1, \, \supp
 \mu \textrm{ is infinite and bounded}\}$ that maps the Jacobi parameters $\{a_n,b_n\}_{n=1}^\infty$ to the measure whose recursion parameters in \eqref{1.3} are the given Jacobi parameters.  Again, there are bijections, $J_n$, mapping $\{a_j\}_{j=1}^{n-1}\cup \{b_j\}_{j=1}^n$ to $n$--point measures.  The question of continuity is a little subtle.  We note here that restricted to measures with support in $[-R,R]$ and the corresponding set of $a$'s and $b$'s, the map is a homeomorphism when the measures are given the weak topology and the Jacobi parameters the topology described in the last paragraph (with $\alpha_{N-1} \in \bbD$ replaced by $a_{N} = 0$).  We'll say more about the issue of topologies in $\textrm{(g)}$ of Section \ref{s5}.

We'll focus here on two of the simplest and most basic sum rules.  One is the Szeg\H{o}--Verblunsky sum rule for OPUC, basically Szeg\H{o}'s Theorem in
the form written down by Verblunsky \cite{Verb}(who was the first to prove the Theorem when $d\mu_s \ne 0$):
\begin{equation}\label{1.4}
  \sum_{j=0}^{\infty} -\log(1-|\alpha_j|^2) = - \int \log(w(\theta)) \frac{d\theta}{2\pi}
\end{equation}
Critical for applications to spectral theory, all terms in the sum are positive so one can prove and interpret this in all cases even when both sides are infinite.  In particular, the fact that two sides are finite simultaneously implies that
\begin{equation}\label{1.5}
  \sum_{j=0}^{\infty} |\alpha_j|^2 < \infty \iff \int \log(w(\theta)) \frac{d\theta}{2\pi} > -\infty
\end{equation}
\cite{SzThm} uses the term `spectral theory gem' for a result like this where there is a strict equivalence between conditions on the spectral data (i.e.\ the measure) and conditions on the recursion coefficients.

The other is the sum rule of Killip--Simon \cite{KS} for OPRL:
\begin{align}
  Q(\mu) &+ \sum_{ \mu\left(\{E_n\} \right)>0,  \ |E_n|>2} F(E_n) = \sum_{n=1}^{\infty} \left[\tfrac{1}{4}b_n^2 + \tfrac{1}{2} G(a_n) \right] \label{1.6} \\
\textrm{where} \nonumber\\
  Q(\mu) &= \frac{1}{4\pi}\int_{-2}^{2} \log\left(\frac{\sqrt{4-x^2}}{2\pi w(x)}\right) \sqrt{4-x^2} \, dx \label{1.7} \\
  F(E) &= \tfrac{1}{2} \int_{2}^{|E|} \sqrt{x^2-4} \, dx \label{1.8} \\
  G(a) &= a^2 - 1 - \log(a^2) \label{1.9}
\end{align}
with the condition that $Q(\mu) = \infty$ unless $\{x \,| \, w(x) \ne 0\} = [-2,2]$ (up to sets of Lebesgue measure zero).

Again, all terms in \eqref{1.6} are positive.  Moreover $G(a) = 2(a-1)^2+\textrm{O}((a-1)^3), \, F(E) = \tfrac{2}{3}(|E|-2)^{3/2}+\textrm{O}((|E|-2)^{5/2})$ which means that \eqref{1.6} implies the gem (Killip--Simon Theorem) that
\begin{equation}\label{1.8a}
\begin{split}
  \sum_{n=1}^{\infty} (a_n-1)^2 &+ b_n^2 < \infty \\
&\iff \\
 \textrm{ess supp }(d\mu) = [-2,2], \,
   Q(\mu) &< \infty   \textrm{ and } \sum_{m} (|E_m|-2)^{3/2} < \infty
\end{split}
\end{equation}
where $\textrm{ess supp }(d\mu)$ is the support of $d\mu$ with isolated atoms removed.

There are many proofs of Szeg\H{o}'s Theorem and so of \eqref{1.4} (see Chapter 2 of \cite{OPUC1}) but until recently, essentially one proof of the Killip--Simon sum rule -- the original proof of \cite{KS} or variants (e.g. \cite{SiMero} or \cite{SzThm}).  There is a simple analog of this proof for Szeg\H{o}'s Theorem (see \cite[Section 2.7]{SzThm}).

This proof starts by showing a so--called step--by--step sum rule.  One then uses positivity of the terms and semi--continuity of the entropy (an idea appearing already in Verblunsky \cite{Verb}) to get the full sum rule.  The step--by--step rules come from a Jensen formula (e.g. \cite[Section 9.8]{CAA}) for suitable functions on the disk.  The value and derivatives at zero involve recursion coefficients and the integral over $\partial\bbD$ the measure side. By taking derivatives at $0$ of a Poisson--Jensen formula, one actually gets an infinite set of step--by--step sum rules related to work of Case \cite{Case} in the Jacobi matrix situation.  The sum rule involving the $n$th derivative is called $C_n$ by Killip--Simon.  None of the $C_n$ sum rules have positive terms but Killip--Simon discovered that $C_0 + \tfrac{1}{2} C_2$ has positive terms and that yielded \eqref{1.6}.  For a long time, there was no explanation for why this turned out to be positive other than good luck.  Work of Nazarov et al \cite{NPVY} (see Denissov--Kupin \cite{DK} for OPUC) shed some light on the positivity condition but it was still unclear why certain combinations of the Taylor coefficients involved are positive. Moreover, the functions involved in the sum rules (e.g., $Q, F,$ and $G$ of \eqref{1.7}--\eqref{1.9} above) seem to simply `pop out' of these combinations and their significance was generally unclear.

Recently Gamboa, Nagel and Rouault \cite{GNR1} (henceforth GNR; they rely on an earlier work of two of those authors \cite{GR1} and have further papers \cite{GNR2,GNR3}) changed this situation.  Using the theory of large deviations, they found a new proof of the Szeg\H{o}--Verblunsky and Killip--Simon sum rules.  One punch line of their work is that the Szeg\H{o}--Verblunsky sum rule follows from a large deviation principle for the spectral measures of the Circular Unitary Ensemble (CUE) of random matrix theory and that the Killip--Simon sum rule follows from a large deviation principle for the spectral measures of the Gaussian Orthogonal Ensemble (GOE). The paper \cite{GNR1} is written for experts in probability theory and is somewhat opaque to spectral theorists. This expository paper is intended to give spectral theorists and people working on orthogonal polynomials access to this work of importance to them.  It may also be useful to probabilists who want to understand the context of the results.

Ignoring technical issues for now, a family of measures, $\{d\bbP_N\}_N$ on a space, $X$, is said to obey a large deviation principle (LDP) if the probability that $x$ is near $x_0$ goes roughly as $e^{-NI(x_0)}$ as $N \to \infty$.  $I(x)$ is called the rate function
and is necessarily non-negative.  Suppose now that $\{d\bbP_N\}$ is a family of probability measures on the set of probability measures on $\partial\bbD$.  Then the bijection $V$ drags these measures to a set of measures, $\widetilde{d\bbP}_N$, on the set of possible Verblunsky coefficients (i.e. on $\prod_{j=0}^{\infty} \bbD \cup \bigcup_{k=1}^\infty (\bbD^{k-1} \cup \partial\bbD)$).  Since LDP is defined in terms of ``nearby'' and $V$ is a homeomorphism, the $d\bbP_N$ obey a LDP with rate function, $I_M(\mu)$, if and only if the $\widetilde{d\bbP}_N$ obey a LDP with rate function $I_V(\alpha)$ where
\begin{equation}\label{1.10}
  I_V(\alpha) = I_M(\mu) \qquad \quad \textrm{whenever } \mu = V(\alpha)
\end{equation}
$V(\alpha)$ is the measure associated to $\alpha$ so \eqref{1.10} is precisely an equality of a function of the recursion coefficients and a function of its associated measure, i.e.\ a sum rule.  Moreover, since
a rate function is necessarily non-negative,
both sides of \eqref{1.10} are positive, so this is a way of generating \emph{positive} sum rules.

This approach illuminates why the various quantities in the Killip--Simon sum rule, that in their proof occur by ad hoc considerations, arise naturally.  
The $\sqrt{4-x^2}/2\pi$ in their quasi--Szeg\H{o} condition (i.e., in $Q$ of \eqref{1.7}) is just the Wigner semi--circle density for GOE.  The function $G$ is just the large deviations rate function for averages of exponential random variables (see \eqref{2.7B} below) and the function $F$ is just the external potential for Coulomb interaction in an external quadratic field (see (d) of Section \ref{s5}).

In Section \ref{s2}, we present a quick overview of the theory of large deviations.  In Section \ref{s3}, we'll compute the measure side of the LDP that yields the Szeg\H{o}--Verblunsky sum rule and in Section \ref{s4},
the coefficient side.  Section \ref{s5}
describes both sides of the LDP that leads to the Killip--Simon sum rule and
Section \ref{s6} has some remarks on further developments.
The Appendix presents an inductive proof to a formula by Killip and Nenciu (\cite{KN} and Theorem \ref{T4.2} below) regarding the probability distribution of the Verblunsky coefficients of a Haar distributed unitary matrix. 

We thank Peter Yuditskii for telling two of us (JB and BS) about \cite{GNR1} and Fritz Gesztesy for encouraging us to write a pedagogic note.

\section{Large Deviations} \lb{s2}

Here we'll sketch some of the key ideas in the theory of large deviations.  Two books on the subject are Deuschel--Stroock \cite{DS} and Dembo--Zeitouni \cite{DZ}.  As both books point out, there is not so much a theory as a collection of powerful tools.  While the subject has roots going back to Laplace, the modern framework goes back to Varadhan, Donsker--Varadhan and
Freidlin--Wentzel in the 1960's and 1970's.

Let $\{\bbP_N\}_{N=1}^\infty$ be a sequence of probability measures on a Polish space, $X$ (see Simon \cite[Section 4.14]{RA} or Billingsley \cite{Bill} for measure theory on Polish spaces).  Let $I$ be a non--negative function on $X$ and $\{a_N\}_{N=1}^\infty$ a sequence of positive numbers with $a_N \to \infty$.  We say that $\{\bbP_N\}_{N=1}^\infty$ obeys a LDP with rate function, $I$, and speed $\{a_N\}_{N=1}^\infty$ if and only if

\begin{SL}
  \item[(1)] $I$ is non-negative and lower semicontinuous on $X$

  \item[(2)] For every open set, $U \subset X$, we have that
  \begin{equation}\label{2.1}
    \liminf_{N \to \infty} \frac{1}{a_N} \log \bbP_N(U) \ge - \inf_{x \in U} I(x)
  \end{equation}

  \item[(3)] For every closed set, $K \subset X$, we have that
  \begin{equation}\label{2.1closed}
    \limsup_{N \to \infty} \frac{1}{a_N}\log \bbP_N(K) \le - \inf_{x \in K} I(x)
  \end{equation}
\end{SL}

We note that the rate function is uniquely determined by these conditions because lower semicontinuity implies that $I(x_0)
= \lim_B \inf_{x \in B} I(x)$ where the limit is over the directed set of all open (or over all closed) neighborhoods of $x_0$, directed by inverse inclusion.  A rate function is called \emph{good} if for each positive $\alpha$, $\{x \,| \, I(x) \le \alpha\}$ is compact (note that by the lower semicontinuity, this set is always closed).  The following two elementary results, whose proofs we leave to the readers, will be useful.

\begin{theorem} \lb{T2.1} Let $a_N = N^\ell$ (for $\ell > 0$).  Fix $N_0$.  Then $\{\bbP_{N+N_0}\}_{N=1}^\infty$ obeys a LDP with speed $a_N$ and rate function, $I$, if and only if $\{\bbP_{N}\}_{N=1}^\infty$ does.
\end{theorem}

\begin{remark} All that is needed is that $\lim_{N \to \infty} \frac{a_{N+N_0}}{a_N} = 1$.
\end{remark}

\begin{theorem} \lb{T2.2} Let $U \subset \bbR^\nu$ be open.  Let $G$ be continuous on $U$ with $\lim_{x \to \partial U \cup \{\infty\}} G(x) = \infty$ and $\inf_{x \in U} G(x) = 0$.  Let $F \in L^1(\bbR^\nu)$ with $\supp  F \subset \overbar{U}, F\ge 0$ and $\inf_{x \in K} F(x) > 0$ for all compact $K \subset U$.  Let $d\bbP_N(x) = Z_N^{-1} e^{-NG(x)} F(x) d^\nu x$ where $Z_N = \int  e^{-NG(x)} F(x) d^\nu x$.  Then $\bbP_N$ obeys a LDP with speed $N$ and good rate function $G$.
\end{theorem}

\begin{remark} The assumptions are overly strong but suffice for the applications we make below.
\end{remark}

A basic result is:

\begin{theorem} [Cram\'{e}r's Theorem] \lb{T2.3} Given a real random variable, $\xi$, let
\begin{equation}\label{2.3}
  \Lambda(\lambda) = \log \mathbb{E}(e^{\lambda\xi})
\end{equation}
be its cumulant generating function and
\begin{equation}\label{2.4}
  I(x)  = \sup_{\lambda \in \bbR} \left(\lambda x - \Lambda(\lambda)\right)
\end{equation}
its Legendre transform.  Let $\bbP_N$ be the probability distribution for $N^{-1} S_N \equiv N^{-1} (X_1+\cdots +X_N)$, where $\{X_j\}_{j=1}^\infty$ are independent copies of $\xi$.  Then $\bbP_N$ obeys a LDP with speed $N$ and good rate function $I$.
\end{theorem}

\begin{remark} By Jensen's inequality, $\Lambda$ is convex and, as a $\sup$ of linear functions, so is $I(x)$.  If $\Lambda$ is everywhere finite, it is $C^1$, $I(\bar{x}) = 0$ where $\bar{x} =\mathbb{E}(\xi)$ and $I(x) > 0$ for $x \ne \bar{x}$.  Thus this result amplifies the law of large numbers.
\end{remark}

\begin{proof} [Sketch]  We sketch the proof in case $\Lambda(\lambda)<\infty$ for all $\lambda$ and the image of $\xi$ is the whole real line; see \cite{DZ} for the full result. Under these conditions, a bit of calculus and convex analysis shows that $I(x)$ is non-negative, strictly convex,
with $I(\bar x)=0$ where $\bar x=\mathbb{E} (\xi)$, and further that
\begin{equation} \lb{2.5}
	I(x)=x\lambda_x-\Lambda(\lambda_x)\, \mbox{\rm
		with} \, \Lambda'(\lambda_x)=x\, \mbox{\rm and
		$\lambda_x \geq 0$ iff $x\geq \bar x$}
\end{equation}
	(The key point in \eqref{2.5} is that a solution to the equation $\Lambda'(\lambda_x)=x$ exists; it is mostly here that we use the assumptions that the image of $\xi$ is $\bbR$ and that $\mbox{\rm Dom}(\Lambda)=\mathbb{R}$.)

	To see the large deviations upper bound, we note first that the strict convexity of $I(\cdot)$ implies that the latter is strictly monotone increasing (resp. decreasing) on $[\bar x,\infty)$ (resp. $(-\infty,\bar x]$)). It is therefore enough to 	prove the upper bound on intervals of the form
$(-\infty,x]$ (with $x<\bar x$) or $[x,\infty)$ (with $x>\bar x$). Considering the latter, we have, with $\lambda\geq 0$,
\begin{equation*}
  \bbP(S_n/n\geq x)\leq \bbE(e^{-n\lambda x+\lambda S_n})= e^{-n(\lambda x-\Lambda(\lambda))}
\end{equation*}
where the independence of the $X_i$'s was used in the last equality.  Choosing $\lambda=\lambda_x$ completes the proof of the large deviations upper bound.

To see the lower bound, it is enough to show that
	\begin{equation} \label{2.6}
		\lim_{\delta\to 0} 	\liminf_{n\to\infty} \frac{1}{n} \log \bbP(S_n/n \in (x-\delta,x+\delta))
	=-I(x).
\end{equation}
To see the latter, introduce the probability distribution $\nu$ by setting
$d\nu/d\mathbb{P}_1(y)=e^{\lambda_x y-\Lambda(\lambda_x)}$.  Then,
\begin{align*}
 \bbE_\nu(X_1) &= \int y e^{\lambda_x y - \Lambda(\lambda_x)} d\mathbb{P}_1(y) \\
                &= e^{-\Lambda(\lambda_x)} \frac{d}{d\lambda} \left. e^{\Lambda(\lambda)} \right|_{\lambda=\lambda_x} \\
                &= \Lambda'(\lambda_x) = x
\end{align*}
so, by \eqref{2.5},
\begin{align}
	\mathbb{P}(S_n/n\in (x-\delta,x+\delta))&= \int_{\sum_{i=1}^n x_i/n\in (x-\delta,x+\delta)} \prod_{i=1}^n	d\mathbb{P}_1(x_i) \nonumber\\
	                               &= \int_{\sum_{i=1}^n x_i/n\in (x-\delta,x+\delta)} e^{-\lambda_x \sum_{i=1}^n x_i+n\Lambda(\lambda_x)} \prod_{i=1}^n	d\nu(x_i) \nonumber\\
	                               &\geq e^{-n(\lambda_x\delta +x \lambda_x- \Lambda(\lambda_x))} \int_{\sum_{i=1}^n x_i/n\in (x-\delta,x+\delta)} \prod_{i=1}^n	d\nu(x_i) \nonumber\\
	                               &= e^{-n(\lambda_x\delta +I(x))} \int_{\sum_{i=1}^n x_i/n\in (x-\delta,x+\delta)} \prod_{i=1}^n	d\nu(x_i) 	\label{2.7}
\end{align}
Since $E_\nu(X_1)=x$, the law of large numbers implies that for any $\delta>0$, the last integral in \eqref{2.7} converges to $1$ as $n\to\infty$.
Taking the limits $n\to\infty$ first and then $\delta\to 0$ completes the proof of the lower bound.
\end{proof}

\begin{example} \lb{E2.3A} Relevant to our considerations later is the average of exponential random variables.  So let $\{X_j\}_{j=1}^\infty$ be independent, identically distributed random variables (iidrv) with density $\chi_{[0,\infty)}(x) e^{-x} dx$.  The cumulant generating function is
\begin{equation}\label{2.7A}
  \Lambda(\lambda) = \log\left(\int_{0}^{\infty} e^{\lambda x} e^{-x} \, dx \right) =
  \begin{cases}
    -\log(1-\lambda), & \mbox{if } \lambda<1 \\
    \infty, & \mbox{if } \lambda \ge 1
  \end{cases}
\end{equation}
For $x \le 0$, taking $\lambda \to -\infty$ in $\lambda x - \Lambda(\lambda)$, we see that $I(x) = \infty$.  If $x>0$, the $\lambda$ derivative of $\lambda x - \Lambda(\lambda)$ vanishes at $\lambda = 1 - x^{-1}$ at which point $\lambda x - \Lambda(\lambda)$ has the value $x-1-\log(x)$.  Thus
\begin{equation}\label{2.7B}
\varphi(x) \equiv I(x) =
\begin{cases}
  x-1 - \log(x), & \mbox{if } x>0 \\
  \infty, & \mbox{if } x \le 0
\end{cases}
\end{equation}
is the (good) rate function.  It is no coincidence as we'll see that the function of \eqref{1.9} is $G(a) = \varphi(a^2)$.  We summarize the combination of this calculation and Cram\'{e}r's Theorem in the theorem below.  The gamma distribution with parameters $\alpha, \beta$ ($\alpha, \beta>0$) is the measure given by 
\begin{equation}\label{2.7Ba}
  dG_{\alpha,\beta}(x) = \frac{\beta ^{\alpha }x^{\alpha -1}e^{-x\beta }}{\Gamma (\alpha )}\,\chi_{[0,\infty)}(x)\,dx
\end{equation}
For exponential iidrv, $n^{-1}\sum_{j=1}^{n} X_j$ has distribution $G_{n,n}$, so this example allows one to also read off a LDP for suitable gamma distributions.
\end{example}

\begin{theorem} \lb{T2.3B} Let $\ell_N$ be integers with
\begin{equation}\label{2.7C}
  \lim_{N \to \infty} N^{-1} \ell_N = \alpha > 0
\end{equation}
Then $Y_N \equiv N^{-1} \sum_{j=1}^{\ell_N} X_j$ with $X_j$ iid exponential random variables obeys a LDP with speed $N$ and rate function
\begin{equation}\label{2.7D}
  \varphi_\alpha (y) \equiv \alpha \varphi(y/\alpha) = y - \alpha - \alpha \log(y/\alpha)
\end{equation}
\end{theorem}

\begin{remark} This goes beyond the direct use of Cram\'{e}r in two ways.  First, we note that if real valued $Z_N$ have a LDP with speed $N$ and rate $I$, then $\alpha Z_N$ has a LDP with speed $\alpha N$ and rate $\alpha I(\cdot/\alpha)$ by a trivial calculation.  Secondly, if $\alpha_N = \ell_N/N$ and $\alpha_N \to 1$, then $\alpha_N^{-1}Y_N$ has a LDP with speed $N$ and rate $I$ if $Y_N$ does and the rate function is continuous.
\end{remark}

Next, we discuss a result known as the contraction principle which allows one to pull a LDP over under continuous maps.  For our basic situation, the maps are homeomorphisms so it is trivial that LDP's carry over but in two places below we'll need the following:

\begin{theorem} [Contraction Principle] \lb{T2.3A} Let $X$ and $Y$ be Polish spaces and $f: X \to Y$ a continuous function onto $Y$.
Suppose $\{\bbP_N\}_{N=1}^\infty$ is a family of probability measures on $X$ that obeys a LDP with speed $a_N$ and good rate function $I$. Define on $Y$ the function
\begin{equation}\label{2.7AA}
  I^{(f)}(y) = \inf \{I(x) \, | \, f(x)=y\}
\end{equation}
 Then the family of measures on $Y$ defined by
\begin{equation}\label{2.7Bcont}
  \bbP^{(f)}_N(A) = \bbP_N(f^{-1}[A])
\end{equation}
obeys a LDP with speed $a_N$ and good rate function $I^{(f)}$.
\end{theorem}

\begin{proof} A simple argument shows that $I^{(f)}$ is a good rate function (see \cite[Section 4.2.1]{DZ}).  If $A$ is open (resp.\ closed), so is $f^{-1}[A]$ and
\begin{equation}\label{2.7CC}
  \inf_{y \in A} I^{(f)}(y) = \inf_{x \in f^{-1}[A]} I(x)
\end{equation}
so the LDP bounds for $\bbP_N$ carry over to such bounds for $\bbP^{(f)}_N$.
\end{proof}

The last topic that we want to consider in this section is the theory of projective limits of LDP's or at least a very special case -- projective limits are indexed by directed sets; we'll only consider the case where the directed set is  $\bbZ_+$.  We have Polish spaces $\{X_j\}_{j=1}^\infty$ and $X$ and
continuous maps $\pi_j:X \to X_j$ and $\pi_{j+1,j}:X_{j+1} \to X_j$ all onto so that $\pi_{j+1,j}\pi_{j+1}=\pi_j$. We require that if $\pi_j(x) = \pi_j(y)$ for $x,y \in X$ and all $j$, then $x=y$. (In the abstract discussion, one just needs to be given $X_j$ and $\pi_{j+1,j}$ and can form $X$ as the subset of $\prod_{j=1}^\infty X_j$ of those $x=(x_j)$ with $\pi_{j+1,j}(x_{j+1})=x_j$.  One puts the product topology on $X$.  In the cases of interest, $X$ will be explicitly given but it agrees with this abstract construction.)  For a measure $\bbP$ on $X$, define $\pi_j^*(\bbP)$, a measure on $X_j$ by $\pi_j^*(\bbP)(A) = \bbP(\pi_j^{-1}[A])$.  Here is a basic theorem due to Dawson--G\"{a}rtner \cite{DG} (see \cite{DZ} for a proof):

\begin{theorem} [Projective Limit Theorem] \lb{T2.4} Let $\{\bbP_N\}_{N=1}^\infty$ be a family of measures on $X$.  Suppose that for each $j$, $\{\pi^*_j(\bbP_N)\}_{N=1}^\infty$ obeys a LDP with speed $a_N$ and good rate function, $I_j$ on $X_j$.  Let
\begin{equation}\label{2.8}
  I(x) = \sup_j \{I_j(\pi_j(x))\}
\end{equation}
Then I is a good rate function and $\{\bbP_N\}_{N=1}^\infty$ obeys a LDP with speed $a_N$ and rate function $I$.
\end{theorem}

\begin{remarks} 1. The converse, i.e.\ if $\{\bbP_N\}_{N=1}^\infty$ obeys a LDP then so does each $\{\pi^*_j(\bbP_N)\}_{N=1}^\infty$, is trivial by the contraction principle.

2.  The same idea shows that if $\{\pi^*_{j+1}(\bbP_N)\}_{N=1}^\infty$ obeys a LDP, so does $\{\pi^*_j(\bbP_N)\}_{N=1}^\infty$ and
\begin{equation} \lb{2.12}
  I_j(x) = \inf \{I_{j+1}(y) \, | \, \pi_j(y) = x\}
\end{equation}
which shows that $I_j(\pi_j(x))$ is monotone in $j$ so the $\sup$ in \eqref{2.8} is a limit.
\end{remarks}

\begin{example} [$\bbR^\infty$] \lb{E2.8} Take $X_j=\bbR^j$, $X=\bbR^\infty=\{(x_1,x_2,\dots)\,|\, x_j \in \bbR\}$ which is a Polish space and $\pi_j(x)_k = x_k$ for $k=1,\dots,j$.  Theorem \ref{T2.4} says that to prove a LDP for $X$, we need only prove it for the finite dimensional $\bbR^j$.
\end{example}

\begin{example} [$\calM_{+,1}(\partial\bbD)$]  \lb{E2.9} Let $\bbP$ be a measure on $\calM_{+,1}(\partial\bbD)$, the probability measures on the unit circle.  Given $\mu \in \calM_{+,1}(\partial\bbD)$ and $j = 1,2,\dots$, let $\pi_j(\mu)$ be the point in $\bbR^{2^j}$ with coordinates $\mu(I_k^{(j)}), k=1,\dots,2^j$ where
\begin{equation}\label{2.13}
  I_k^{(j)} = \{e^{2\pi i \theta} \, | \, \tfrac{k-1}{2^j} \le \theta < \tfrac{k}{2^j} \}
\end{equation}
Realizing $\bbR^{2^j}$ as a set of measures, we can think of
\begin{equation}\label{2.13A}
  \pi_j(\mu) = \sum_{k=1}^{2^j} \mu(I_k^{(j)}) 2^j \chi_{I_k^{(j)}}(x) \, dx
\end{equation}
Thus $\bbP$ induces a measure $\pi_j^*(\bbP)$ on either $\bbR^{2^j}$ or on $\calM_{+,1}(\partial\bbD)$ supported on a $2^j$--dimensional subspace.  The $\pi_j(\mu)$ determine $\mu(\{e^{2\pi i \theta} \, | \, 0 \le \theta < \tfrac{k}{2^j} \})$ and so $\mu$.  Clearly as $j \to \infty$, $\pi_j(\mu)$ converges weakly to $\mu$.

In this case
\begin{equation}\label{2.14}
  \pi_{j+1,j}(y)_\ell = y_{2\ell -1}+y_\ell \quad \ell = 1,\dots,2^j
\end{equation}
Thus, to get a LDP for $\calM_{+,1}(\partial\bbD)$, we need only prove $2^j$--dimensional LDPs.
\end{example}

\section{Szeg\H{o}'s Theorem: Measure Side}  \lb{s3}

We begin our presentation of the proof of Szeg\H{o}'s theorem using large deviations for CUE.  CUE($n$) is just another name for Haar measure on the unitary $n \times n$ matrices.  Let $\{e_j\}_{j=1}^n$ be the standard basis for $\bbC^n$.  It is easy to see that for a.e.\ $U$, $e_1$ is a cyclic vector for $U$ so that $U$ and $e_1$ define a spectral measure
\begin{equation}\label{3.1}
  d\mu(\theta) = \sum_{j=1}^{n} w_j \delta_{\lambda_j}
  \end{equation}
on $\partial\bbD$, with precisely $n$ pure points (aka atoms) $\lambda_j = e^{i\theta_j}, j=1,\dots,n$. Letting $\{\varphi_j\}_{j=1}^n$ be the orthonormal basis of eigenvectors of $U$, so that $U\varphi_j = \lambda_j \varphi_j$, we have $w_j = |\jap{\varphi_j,e_1}|^2$.  Of course, since $\norm{e_1} = 1$,
\begin{equation}\label{3.2}
  \sum_{j=1}^{n} w_j = 1
\end{equation}

Thus Haar measure induces a measure $\bbP_n$ on $n$--point probability measures (which we can think of as a measure on all measures that happens to be supported on the n--point measures) which we also call CUE($n$).

For $\tilde U$ an arbitrary unitary, $\tilde UU\tilde U^{-1}$
has the same eigenvalues as $U$ and $\jap{\varphi_j(\tilde U
  U\tilde U^{-1}),e_1} = \jap{\tilde U\varphi_j(U),e_1}$.
  Since $U \mapsto \tilde UU\tilde U^{-1}$ leaves Haar measure invariant,
  we see that the distribution of the unit vector
  $\left( \jap{\varphi_1(U),e_1}, \jap{\varphi_2(U),e_1}, \ldots, \jap{\varphi_n(U),e_1} \right) \in \bbC^n$ is invariant under unitary transformations, which implies it is the Euclidean measure restricted to the sphere. By using the fact that that $d^2z = \tfrac{1}{2} d\theta d(|z|^2)$ (which shows it is essential we work in $\bbC$), it is not hard to show that the squares of the components of a complex $n$--vector uniformly distributed on the sphere are uniformly distributed on the simplex  \eqref{3.2}. Thus we get that the $\{w_j\}_{j=1}^n$ are independent of the eigenvalues and have $\bbP_n$-distribution
\begin{equation}\label{3.3}
  (n-1)! \chi_{\{\sum_{j=1}^{n-1}w_j\ \le 1 \, ; \, w_j \ge 0\}} (w) dw_1 \dots dw_{n-1}
\end{equation}

The distribution of the eigenvalues is given by the Weyl integration formula which says that the distribution of the eigenvalues under Haar measure is
\begin{equation} \lb{3.4}
 \frac{1}{n!} |\Delta(e^{i\theta_1},\dots,e^{i\theta_n})|^2 \prod_{j=1}^{n} \frac{d\theta_j}{2\pi}
\end{equation}
\begin{equation} \lb{3.5}
  \Delta(\lambda_1,\dots,\lambda_n) \equiv \prod_{i<j} (\lambda_i - \lambda_j)
\end{equation}
For proofs of this formula from two different points of view, see Anderson et al \cite[Section 4.17]{AGZ} or Simon \cite[Section IX.3]{SiGpRep}.  Thus
\begin{align}
     d\bbP_n(\theta_1,\dots,\theta_n,w_1,\dots,w_n) = \frac{1}{n (2\pi)^n}  \chi_{\{\sum_{j=1}^{n-1}w_j\ \le 1 \, ; \, w_j \ge 0\}} (w) \nonumber \\
     \quad |\Delta(e^{i\theta_1},\dots,e^{i\theta_n}|^2 d\theta_1 \dots d\theta_n \, dw_1 \dots dw_{n-1} \lb{3.6}
\end{align}

With this in hand, we turn to the proof of the measure theory LD result:

\begin{theorem} \lb{T3.1} $\bbP_n$ as a family of measures on the set of probability measure on $\partial\bbD$ obeys a LDP with speed $n$ and good rate function
\begin{equation}\label{3.7}
  I(\mu) = - \int \log w(\theta) \, \frac{d\theta}{2\pi}
\end{equation}
where $w$ is given by \eqref{1.1}
\end{theorem}

\begin{remarks} 1. If $\mu,\nu$ are two probability measures, then
\begin{equation} \lb{3.8}
S(\nu \, | \,\mu) = \begin{cases}
  -\int \log \left(\frac{d\nu}{d\mu}\right) \, d\nu , & \mbox{if } \nu \mbox{ is } \mu\mbox{--a.c.} \\
  -\infty, & \mbox{otherwise}.
\end{cases}
\end{equation}
is called the relative entropy.  Its negative is called the Kullback--Leibler (KL) divergence by statisticians.  Both sign conventions are used for ``the relative entropy''.  We follow much of the spectral theory literature which follows Killip--Simon \cite{KS}.  Most of the probability literature uses the opposite sign! \eqref{3.7} is $-S(\tfrac{d\theta}{2\pi}\,|\,\mu)$.  That $I$ is lower semicontinuous goes back to Verblunsky \cite{Verb}.  That the reverse Kullback--Leibler divergence is relevant to large deviations
appears in Lynch--Sethuraman \cite{LS} and Ganesh--O'Connell \cite{GO}.

2. It is remarkable (though not unusual) that for each $n$, $\bbP_n$ is supported on a set where the rate function (i.e. the Kullback-Leibler) divergence is infinite. This is since the approximating measures are discrete while the limiting one is absolutely continuous.
\end{remarks}

As a preliminary, one needs to look at what spectral theorists call the density of states, OP workers the density of zeroes and probabilists the empirical measure, namely
\begin{equation}\label{3.9}
  \mu^{(E)} = \frac{1}{n} \sum_{j=1}^{n} \delta_{\lambda_j}
\end{equation}
where $\lambda_j$ are the atoms of $\mu$ in \eqref{3.1}. $\bbP_n$ induces a distribution $\bbP_n^{(E)}$ on point measures of the form \eqref{3.9}, essentially given by \eqref{3.4}/\eqref{3.5}.  One has the following result of Ben Arous and Guionnet \cite{BAG} (see also \cite[Section 2.6]{AGZ}; these results discuss GUE, not CUE -- the analog for CUE uses the same ideas and is even simpler):

\begin{theorem} \lb{T3.2} $\bbP_n^{(E)}$ obeys a LDP with speed $n^2$ and good rate function
\begin{equation}\label{3.10}
  I(\mu) = - \int \log(|z-w|) \, d\mu(z) \, d\mu(w)
\end{equation}
\end{theorem}

\begin{remark} In \eqref{3.10}, $z$ and $w$ lie in the unit circle and $|z-w|$ is a two dimensional distance.  This is a $2D$ Coulomb energy.  There is a close connection between this result and Johansson's proof \cite{Joh} of the Strong Szeg\H{o} Theorem.
\end{remark}

We will not give a formal proof of Theorem \ref{T3.2} but instead indicate the basic intuition. For distinct $\lambda_i$s,
\begin{align}
  \prod_{i<j} |e^{i\theta_i} - e^{i\theta_j}|^2 &= \exp\left(-n^2 J_n(\lambda_1,\dots,\lambda_n)\right) \lb{3.11} \\
  J_n(\lambda_1,\dots,\lambda_n)                &= -\frac{2}{n^2} \sum_{i<j} \log(|\lambda_i - \lambda_j|) \nonumber \\
                                                &= -\frac{1}{n^2} \sum_{i \ne j} \log(|\lambda_i - \lambda_j|) \lb{3.12}
\end{align}
If $\mu^{(E)}$ is an $n$--point measure near $\mu$ and the $\lambda$ have reasonable local spacing, the sum in \eqref{3.12} should be near the integral in \eqref{3.10}.  This completes our description of the intuition behind the proof of Theorem \ref{T3.2}.

The weights and eigenvalues are independent. We'll consider a fixed triangular array of eigenvalues $\{\lambda_\ell^{(n)}\}_{1 \le \ell \le n; \, n=1,\dots}$ where we suppose that
\begin{equation}\label{3.13}
  \frac{1}{n} \sum_{\ell=1}^{n} \delta_{\lambda_\ell^{(n)}} \to \frac{d\theta}{2\pi}
\end{equation}
weakly. We distribute weights uniformly on the simplex and look at
\begin{equation} \lb{3.12A}
\{w_\ell\}_{\ell=1}^n \mapsto \sum_{\ell=1}^{n} w_\ell \delta_{\lambda_\ell^{(n)}} \equiv \mu_n(w_\ell)
\end{equation}
This gives a distribution, $\bbP_n^{(\lambda)}$, on measures and we'll prove these measures obey a LDP with speed $n$ and rate function $I$ given by \eqref{3.7}.  A full analysis depends on proving for each $\epsilon > 0, \, j$ and $k=1,\dots,2^j$, the probability that $\left|\tfrac{2^j}{n} \#(\ell \,|\, \lambda_\ell^{(n)} \in I_k^{(j)}) - 1 \right| \ge \epsilon$ (with $I_k^{(j)}$ given by \eqref{2.13}) goes to zero faster than exponentially in $n$. This is where Theorem \ref{T3.2} is used.

The proof will be to use projective limits with the maps of Example \ref{E2.9}. We'll get a LDP for the projections using Example \ref{E2.3A} and control the $\sup$ of the projected rate functions by a general continuity result.  It is this last fact that will show singular parts of the measure only change the rate by their impact on the total weight of the a.c.\ part.

For each $j = 1,\dots$ and $k=1,\dots,2^j$, let $I_k^{(j)}$ be given by \eqref{2.13} and $\pi_j(\mu)$ the measure in \eqref{2.13A}.  Given $\{w_\ell\}_{\ell=1}^{n}$, let $\ti{\mu}^j_n(w_\ell)$ be the measure on $\partial\bbD$ with constant a.c.\ weight on each $I^{(j)}_k$ so that
\begin{equation}\label{3.13A}
  \ti{\mu}^j_n(I^{(j)}_k) = \sum_{\lambda_\ell^{(n)} \in I_k^{(j)}} w_\ell
\end{equation}
Thus in terms of the objects given in \eqref{3.12A} and \eqref{3.13A}, we have that $\pi_j(\mu_n(w_\ell)) = \ti{\mu}^j_n(w_\ell)$.

Let $\wti{\bbP}_n^{(j)}$ be the measure on $\bbR^{2^j}$ using \eqref{3.13A} but where now the $w_\ell$ are replaced by iid exponential random variables, $W_\ell$. Thus, $\wti{\bbP}_n^{(j)}$ is the probability measure for the $\bbR^{2^j}$-valued random variable given by
\beq \no
\beta_k^n=\sum_{\lambda_\ell^{(n)} \in I_k^{(j)}} W_\ell
\eeq
Fix $j$ and take $n \to \infty$.  By Theorem \ref{T2.3B} and \eqref{3.13}, $\wti{\bbP}_n^{(j)}$ obeys a LDP with speed $n$ and rate function at the point $\vec{\beta} \equiv \{\beta_\ell\}_{\ell=1}^{2^j} \in \bbR^{2^j}$
\begin{equation}\label{3.14}
  \varphi(\vec{\beta}) = \sum_{\ell=1}^{2^j} \left[(\beta_\ell - 2^{-j}) - 2^{-j} \log(2^j \beta_\ell) \right]
\end{equation}

Recall that given two probability measures $\mu$ and $\nu$ on the same space, their KL divergence, $H(\mu|\nu)$, is given by the negative of \eqref{3.8}.  Write $\beta_\ell = \beta s_\ell$ with $\beta = \sum_{q=1}^{2^j} \beta_q$ so that $\vec{s}$ lies in a $2^j$-simplex.  Write $\mu_{\vec{s}}$ for the probability measure giving uniform weight $s_k$ to $I_k^{(j)}$ and let $\nu$ be normalized Lebesgue measure on the circle (i.e. $\mu_{\vec{s}}$ for the $\vec{s}$ with equal components, $2^{-j}$).  Then \eqref{3.14} can be rewritten:
\begin{equation}\label{3.15}
  \varphi(\vec{\beta}) = \beta - 1 - \log(\beta) + H(\nu|\mu_{\vec{s}})
  \end{equation}
Note this is the sum of a function of $\beta$ only and a function of the $s$'s only.  This is a consequence of the fact that for independent exponential random variables, $\sum_{k=1}^{N} X_k$ is independent of $\{X_j/\sum_{k=1}^{N} X_k\}_{j=1}^N$.  It makes the use of the contraction principle (which, in general, is already simple), extremely simple.

For fixed $\lambda$'s, let $\bbP_n^{(j)} = \pi_j^* \left(\bbP_n^{(\lambda)} \right)$.  This is just the contraction of $\wti{\bbP}_n^{(j)}$ under the map $G(\vec{\beta}) \equiv \vec{\beta}/\beta$ from $\bbR^{2^j}$ to the $2^j$--simplex.  By the contraction principle (Theorem \ref{T2.4}) and
\begin{equation*}
  \inf_{\beta>0} [\beta - 1 - \log(\beta)] = 0
\end{equation*}
(as it must as the rate function, \eqref{2.7D}, when $\alpha=1$), we see that for each fixed $j$, $\bbP_n^{(j)}$ obeys a LDP with speed $n$ and rate function $H(\nu|\mu_{\vec{s}})$.

Given the projection theorem (Theorem \ref{T2.4}), the following completes the proof of Theorem \ref{T3.1}:

\begin{proposition} \lb{P3.3} Let $\mu$ be an arbitrary probability measure on $\partial\bbD$ and $\nu=\frac{d\theta}{2\pi}$.  Let $\pi_j(\mu)$ be given by \eqref{2.13A}.  Then
\begin{equation}\label{3.16}
  \lim_{j \to \infty} H(\pi_j(\nu)|\pi_j(\mu)) = H(\nu|\mu)
\end{equation}
\end{proposition}

\begin{remarks} 1.  $\pi_j(\nu) = \nu$ for this $\nu$.  We write it this way because with a slight change in the proof, it holds for any $\nu$ (and $\mu$).  This extended version is needed for the Killip--Simon theorem and other cases where the limiting empirical measure is not unweighted Lebesgue measure.

2.  By slightly expanding the argument, one sees that $H(\pi_j(\nu)|\pi_j(\mu))$ is monotone increasing in $j$.

3.  It is worth repeating that this holds independently of the singular part of $\mu$ and is responsible for the fact that rate is the same for two $\mu$'s with identical a.c.\ parts independently of their singular parts.
\end{remarks}

\begin{proof} By convexity of $y \mapsto -\log y$ and Jensen's inequality, for any positive function $h$ and probability measure $d\eta(y)$, we have that
\begin{equation*}
  -\int \log h(y) \, d\eta(y) \ge -\log \left(\int h(y) \, d\eta(y) \right)
\end{equation*}
Thus if $d\mu$ has the form \eqref{1.1} and $W_k^{(j)} = 2^j \int_{I_k^{(j)}} w(\theta) \, \tfrac{d\theta}{2\pi}$, we have that
\begin{align}
-\int_{I_k^{(j)}} \log(w(\theta)) \, 2^j\frac{d\theta}{2\pi} &= -\log(W_k^{(j)}) - \int_{I_k^{(j)}} \log \left( \frac{w(\theta)}{W_k^{(j)}} \right) \, 2^j \frac{d\theta}{2\pi} \nonumber \\
                                                        &\ge -\log W_k^{(j)} \ge -\log\left( 2^j \mu(I_k^{(j)}) \right) \lb{3.17}
\end{align}
since $\log \left[\int_{I_k^{(j)}} \frac{w(\theta)}{W_k^{(j)} } 2^j \frac{d\theta}{2\pi}\right] = 0$.

On the other hand, $\pi_j(\mu) \equiv \mu^{(j)}$ is an absolutely continuous measure with weight $\sum_{k=0}^{2^j} 2^j \mu(I_k^{(j)}) \chi_{I_k^{(j)}}(\theta)$ so
\begin{equation}\label{3.17A}
  H(\nu|\mu^{(j)}) = - 2^{-j} \sum_{k=0}^{2^j} 2^j  \log\left( 2^j \mu(I_k^{(j)}) \right)
\end{equation}
which by \eqref{3.17} implies that
\begin{equation}\label{3.18}
  H(\pi_j(\nu)|\pi_j(\mu)) \le H(\nu|\mu)
\end{equation}
so, in particular,
\begin{equation}\label{3.19}
  \limsup H(\pi_j(\nu)|\pi_j(\mu)) \le H(\nu|\mu)
\end{equation}

For the other direction, $\pi_j(\mu) \to \mu$ weakly by an easy argument (the uniform closure of functions constant on the $I_k^{(j)}$ for some $j$ includes all continuous functions).  By the lower semicontinuity of $H$ jointly in the two variables (see \cite[Theorem 2.3.4]{OPUC1} which proves that $-H$ is jointly weakly upper semicontinuous)
\begin{equation}\label{3.20}
  H(\nu|\mu) \le \liminf H(\pi_j(\nu)|\pi_j(\mu))
\end{equation}
\end{proof}

\section{Szeg\H{o}'s Theorem: Coefficient Side}  \lb{s4}

In this section, our goal is to use the bijection $V$ from Verblunsky coefficients to measures on $\partial\bbD$ to move the CUE measures $\bbP_n$ from the measure side to measures $\wti{\bbP}_n$ on Verblunsky coefficients.  We'll prove:

\begin{theorem} \lb{T4.1} $\wti{\bbP}_n$ obeys a LDP with speed $n$ and good rate function
\begin{equation}\label{4.1}
  \ti{I}(\alpha) = - \sum_{j=1}^{\infty} \log(1-|\alpha_j|^2)
\end{equation}
\end{theorem}

\begin{remarks}  1. Since $V$ is a homeomorphism, Theorem \ref{T3.1} implies that $\wti{\bbP}_n$ obeys a LDP.  The point is that \eqref{4.1} expresses the rate on the $\alpha$--side.  In fact, we'll independently prove that $\wti{\bbP}_n$ obeys a LDP.  In further applications, it is useful that we only have to prove a LDP on one side.

2. Of course, this result and \eqref{3.7} imply the Szeg\H{o}--Verblunsky sum rule \eqref{1.4}.
\end{remarks}

One part of the proof is an explicit formula for $\wti{\bbP}_n$ found by Killip--Nenciu \cite{KN}.
For a proof, see Killip--Nenciu \cite{KN}, or the Appendix below.
\begin{theorem} \lb{T4.2} $\wti{\bbP}_n$ is supported on the $n$--point $\alpha$'s, i.e. $\alpha_0,\dots,\alpha_{n-2} \in \bbD, \alpha_{n-1} \in \partial\bbD$ and given by
\begin{align}
  d\wti{\bbP}_n(\alpha_0,\dots,\alpha_{n-1}) &= \prod_{j=0}^{n-1} d\kappa_{n-2-j}(\alpha_j)  \lb{4.2} \\
  d\kappa_\ell(\alpha) &= \frac{\ell+1}{\pi} (1-|\alpha|^2)^\ell d^2\alpha &\textrm{ on  } \bbD; \, \ell \ge 0 \lb{4.3} \\
  d\kappa_{-1} (\alpha = e^{i\theta}) &= \frac{d\theta}{2\pi} & \textrm{ on } \partial\bbD \lb{4.4}
\end{align}
\end{theorem}
%
%
The density is thus $\tfrac{(n-1)!}{\pi^n}\prod_{j=0}^{n-2} (1-|\alpha_j|^2)^{n-2-j} d^2\alpha_j$.  The $\prod_{j=0}^{n-2}(1-|\alpha_j|^2)^n = \exp\left[-n(-\sum_{j=0}^{n-2} \log(1-|\alpha_j|^2))\right]$ suggests that the rate function is given by \eqref{4.1} but naively, the $(n-1)!$ looks worrying.  That it isn't comes from the magic of projective limits.

Projective limits are especially simple in the case of independent variables. So look at measures $\{\bbP^\#_N\}_{N=1}^\infty$ on $X = \overbar{\bbD}^\infty$ by putting $\bbP^\#_N$ on sequences with $\alpha_j=0; \, j \ge N$ and distributing $\{\alpha_j\}_{j=0}^{N-1}$ according to $\wti{\bbP}_N$.  Then, if $j$ is fixed and $N > j$, then
\begin{align}
  d\pi^*_j(\bbP^\#_N)(\alpha_0,\dots,\alpha_{j-1}) =& \frac{(N-1)\cdots(N-j)}{\pi^j}\left[\prod_{k=0}^{j}(1-|\alpha_k|^2)\right]^{N-2-j} \nonumber \\
                                                      & \, \quad \prod_{k=0}^{j-1}\left[(1-|\alpha_k|^2)^{j-k}d^2\alpha_k\right] \lb{4.5}
\end{align}
The leading factor is polynomial in $N$, so unimportant for $a_N = N$ LDPs.  Using Theorems \ref{T2.1} and \ref{T2.2}, we see that $\pi^*_j(\bbP^\#_N)$ obeys a LDP at speed $N$ and rate function
\begin{equation}\label{4.6}
  I_j(\alpha_0,\dots,\alpha_{j-1}) = -\sum_{k=0}^{j-1} \log(1-|\alpha_k|^2)
\end{equation}
Since
\begin{equation}\label{4.7}
  \sup_j I_j(\pi_j(\alpha)) = -\sum_{k=0}^{\infty} \log(1-|\alpha_k|^2)
\end{equation}
this proves, on account of Theorem \ref{T2.4}, that $\{\bbP_N^\#\}_{N=1}^\infty$ obeys a LDP with speed $N$ and rate function \eqref{4.1}.

The measures we are really interested in aren't on the product space $X$ but on the space, $Y$, which is the union of $\bbD^\infty$ with finite sequences in $\bbD^{N-2} \times \partial\bbD$ with the topology described in Section \ref{s1}.  There is a natural map $f:X \to Y$, that given $\alpha \in X$ maps to $\alpha$ if all $\alpha_j \in \bbD$ and drops the $\alpha_k, \, k>j$ is $\alpha_j$ if the first $\alpha_\ell \in \partial\bbD$.  Then $f$ is continuous and $(\bbP_N^\#)^{(f)} = \wti{\bbP}_N$, so Theorem \ref{T2.3A} completes the proof of Theorem \ref{T4.1}.

\section{The Killip--Simon Theorem} \lb{s5}

The large deviation proof of the Killip--Simon Theorem is similar to the one in Sections 3 and 4 with some changes and additions which we briefly describe.

\begin{SL}
  \item[(a)] One uses GUE instead of CUE (GNR use GOE which differs from GUE by some factors of 2).  Thus the measure on random $n \times n$ self--adjoint matrices has $\{\re M^{(n)}_{ij}\}_{1 \le i \le j \le n}$ and $\{\im M^{(n)}_{ij}\}_{1 \le i < j \le n}$ Gaussian iid with mean zero and $\bbE([M_{ii}^{(n)}]^2) = \bbE([\re M^{(n)}_{ij}]^2)=  \bbE([\im M^{(n)}_{ij}]^2)=n^{-1}$ for any $i<j$.

  \item[(b)] The eigenvalue distribution has $\lambda_j \in \bbR$ with distribution
                 \begin{equation} \lb{5.1A}
                 \left[\prod_{i<j} |\lambda_i - \lambda_j|^2\right] e^{-n\sum_{j=1}^{n} \lambda_j^2}
                 \end{equation}
       so the empirical measure converges to the equilibrium measure in a quadratic external field, i.e.\ the minimizer for $-\int \log |x-y| \, d\mu(x) \, d\mu(y) + 2 \int x^2 \, d\mu(x)$.  It is well--known \cite[Exercise 2.6.4]{AGZ} that the minimizer is the semicircle law $d\nu_0(x) \equiv \pi^{-1} (1-x^2)^{1/2} \chi_{[-1,1]} (x) dx$.  To agree with the Killip--Simon notation, one rescales the matrix so the support is $[-2,2]$.

  \item[(c)] By the argument of \cite{BAG}, the empirical measure converges to $\nu_0$ and by mimicking the proof of Theorem \ref{T3.1}, the contribution of the part of the spectral measure on $[-2,2]$ is just $H(\nu_0|\mu)$.  Thus the weight in the Killip--Simon quasi--Szeg\H{o} integral is exactly the Wigner semicircle weight.

  \item[(d)] As we've seen, a single point in the measure, if the point is in the bulk, involves the increase of $H(\nu|\mu)$ due to the weight having a smaller integral.  But if the point is outside $[-2,2]$, there is a contribution due to the location, $\lambda_0$, of the eigenvalue (the impact of the weight is the same whether the point is in the bulk or outside it).  By looking at the $\log$ of the part of \eqref{5.1A} depending on $\lambda_0$, one sees that the decrease in the eigenvalue density involves $\lambda_0$ interacting with $n$ eigenvalues.  The decrease is approximately $\exp(-nF(\lambda_0))$ where $F$ is the potential in the quadratic external field in the equilibrium measure (this idea is due to Ben Arous et al.\ \cite{BADG}).  It is known that this function is the same as the $F$ in equation \eqref{1.8} (see \cite[eqn 1.13]{APS} or \cite[proof of Thm 3.6]{DKMVZ}) so the Killip--Simon $F$ is just an external field potential.

  \item[(e)] For finitely many eigenvalues outside $[-2,2]$ you just get the sums of single costs since the interaction between eigenvalues is $\textrm{O}(1)$, not $\textrm{O}(n)$.  Handling infinitely many eigenvalues converging to $\pm 2$ requires a careful use of projective limits (see \cite{GNR1}).

  \item[(f)] For the coefficient side, Killip--Nenciu is replaced by earlier results of Dumitriu--Edelman \cite{DE} (whose work motivated Killip and Nenciu) who found the distribution of Jacobi parameters for GUE and GOE.  The $\{b_j\}_{j=1}^n$ are Gaussian (with $\textrm{O}(n)$ widths leading to the $b_j^2$ term in the Killip--Simon sum rule).  The $\{a_j^2\}_{j=1}^{n-1}$ are gamma distributed, essentially behaving like sums of exponential random variables and so we get the $G(a_j)$ terms.  Thus $G$ occurs in the sum rule as the rate function for suitable gamma distributions.

  \item[(g)] There is an issue involving the equality of the two sides of the sum rule that we want to discuss, addressed in a related way in Gamboa-Rouault \cite{GR1}.  The natural setting for the LDP for measures is the space, $X'$, of all probability measures on $\bbR$, and for Jacobi parameters the Polish space $Y' \equiv [\bbR \times (0,\infty)]^\infty$ with finite sequences added to it.  The issue is that the inverse Jacobi map isn't defined for all measures but only those with all moments finite and, in general, this inverse map is many--to--one in certain cases where the measure has unbounded support.  Let $X_k$ be the set of measures supported in $[-k-2,k+2]$ for $k=1,2,\dots$ and $Y_k$ its image under the inverse Jacobi maps.  Let $X = \cup_{k=1}^\infty X_k$ and $Y = \cup_{k=1}^\infty Y_k$.  The rate functions are infinite on the complements of these sets. The Jacobi map is a well defined bijection of $Y$ to $X$ but in the relative topologies is not continuous nor is its inverse continuous!  However, it is a homeomorphism restricted to each $Y_k$.  Moreover since the probabilities under $\bbP_N$ (resp.\ $\wti{\bbP}_N)$) of $X_k$ (resp.\ $Y_k$) go to one exponentially fast in $n$ (with rate going to infinity with $k$), it is not hard to prove LDP's for the restrictions of these measures to $X_k$ and $Y_k$ renormalized (by dividing by the probabilities of $X_k$ and $Y_k$) with the same rate functions.  Thus the fact that $J$ is a homeomorphism of $Y_k$ to $X_k$ lets us conclude equalities of the two rate functions and so the sum rule.
\end{SL}

\section{Further Developments} \lb{s6}

Finally, we want to make a few comments about the strategy of the last three sections applied to other sum rules.  We are aware of four papers on this approach. Gamboa, Nagel and Rouault have two preprints \cite{GNR2, GNR3} besides other results in their original paper \cite{GNR1}.  And the authors of the current paper are preparing a work \cite{BSZ2} using these methods in a wider context.

In \cite[Section 2.8]{OPUC1}, Simon found a sum rule involving $-\int (1-\cos(\theta)) \log(w(\theta)) \, \tfrac{d\theta}{2\pi}$ on the measure side and made a conjecture concerning
\begin{equation}\label{6.1}
  -\int \log(w(\theta)) \, d\eta(\theta)
\end{equation}
where
\begin{equation}\label{6.2}
  d\eta(\theta) = Z^{-1} \prod_{j=1}^{k} (1-\cos(\theta - \theta_j))^{m_j} d\theta
\end{equation}
where $Z$ is a normalization factor to make $d\eta$ into a probability measure.  There developed a huge literature on these so called higher order sum rules for OPUC and OPRL including \cite{DK,GZ,Kupin,LNS,Lukic,Lukic2,NPVY}.

The key to understanding such sum rules (for OPUC) in the context of large deviations is to replace Haar measure, $d\bbP_N$, by
\begin{equation}\label{6.3}
  Z_N^{-1} \exp\left[-N \sum_{j=1}^{N} V(\lambda_j)\right] d\bbP_N
\end{equation}
where $V$ is a function on $\partial\bbD$ and $\{\lambda_j\}_{j=1}^N$ are the eigenvalues.  It is well known (see \cite[Section 2.6]{AGZ}) that when $V$ is nice enough, we will get $d\eta$ as the equilibrium measure if
\begin{equation}\label{6.4}
  V(e^{i\theta}) = 2 \int \log |e^{i\theta} - e^{i\psi}| \, d\eta(\psi)
\end{equation}
In a forthcoming paper \cite{BSZ2}, the current authors study this when $d\eta$ is given by \eqref{6.2}.  In the cases we study, $V(e^{i\theta})$ is a finite linear combination of $\cos(m\theta)$.  In terms of $U$, if $e^{i\theta_j}$ are the eigenvalues, $\sum_{j=1}^{n} \cos(m\theta_j) =  \re (\tr(U^m))$ which one can write in terms of Verblunsky coefficients using the CMV representation of $U$.  We obtain a large deviations proof of the $(1-\cos(\theta))$ sum rule of Simon and the gems of Simon--Zlato\v{s} \cite{SZ}. In addition, we prove a partial special case of a conjecture of Lukic \cite{Lukic} that replaces a wrong conjecture of Simon \cite[Section 2.8]{OPUC1}, providing evidence for Lukic's conjecture.  GNR have a paper \cite{GNR3} that discusses in some detail the case $V(\theta) = \cos(\theta)$ where the random matrix model has been studied by Gross--Witten \cite{GW} whose names GNR apply to the model.  They note that formally the large deviations argument leads to a sum rule but for technical reasons, they aren't able to provide a proof.  By using some results from the theory of OPUC, we do prove sum rules in this and the other cases.

In their original paper \cite{GNR1}, GNR introduce two new results they call magic sum rules by using large deviations on two matrix models.  These models have free parameters and lead to families of sum rules.  In these models the continuous spectrum is an interval $[\alpha,\beta]$.  There are three classes of sum rules where the KL divergence has $H(\nu|\mu)$ with
\begin{equation*}
  d\nu(x) = \chi_{[\alpha,\beta]}(x) (x-\alpha)^{\sigma/2}(\beta-x)^{\tau/2} \, dx
\end{equation*}
with $(\sigma,\tau)$, one of $(1,1), (-1,-1)$ or $(-1,1)$ (or $(1,-1)$). All these new sum rules lead to what one of us calls flawed gems because they have apriori conditions on the Jacobi parameters. This is because the sum rules only hold under these conditions (the restriction comes from the fact that the Jacobi parameters have to be expressible in terms of
other more basic parameters of the underlying model and only certain Jacobi parameters can be expressed that way).  The $(1,1)$ examples lead to gems that are restricted forms of the Killip--Simon Theorem and so not new gems.  The $(-1,-1)$ examples are restricted forms of Szeg\H{o}'s Theorem under the Szeg\H{o} maps (see \cite[Section 13.1]{OPUC2}) and so are not new.  Their $(-1,1)$ examples yield new flawed gems.

There has been considerable literature on proving analogs of the Killip--Simon theorem where $[-2,2]$ is replaced by a finite gap set $\fre = \fre_1 \cup \dots \cup \fre_n$ were the $\fre_j$ are disjoint intervals.  In \cite{DKS}, Damanik et al prove a Killip--Simon rule in the case that each $\fre_j$ has harmonic measure (i.e.\ measure in the equilibrium measure for $\fre$) $1/n$.  They first prove a Killip--Simon sum rule for $[-2,2]$ but with matrix valued measures and then use something they call the magic formula to get a gem for these special sets.  In \cite{GNR2}, GNR use large deviations to get a new proof of the Killip--Simon sum rule on $[-2,2]$ with matrix valued measures and then plug that into the DKS magic formula machine to get a partially new proof of the DKS result for $1/n$ harmonic measure sets.  It would be very interesting to obtain this result directly with large deviations without using the magic formula.

Recently Yuditskii \cite{Yud} proved an analog of the Killip--Simon Theorem for any finite gap set, $\fre$.  It would be very interesting to find a large deviations proof of his result.

There has been very little work on Killip--Simon type theorems for finite gap sets in OPUC.  In \cite{GNR3}, GNR obtain a sum rule and gem for $\fre = \{e^{i\theta} \, | \, \alpha \le \theta \le 2\pi - \alpha\}$ for $0 < \alpha < \pi$.  For real $\alpha$, the Verblunsky side has the expected $\sum |\alpha_j - a|^2$ form but for general $\alpha$, it has the form $\sum |\gamma_j - a|^2$ where $\gamma_j$ is a non--local function of the $\alpha$'s.  In particular, it is not clear if the finiteness of their Verblunsky side only depends on the behavior near $j = \infty$.  At least for the real case, it would be interesting to get the sum rule via the Poisson--Jensen methods used in the original Killip--Simon proof \cite{SiMero}.  It would also be interesting to understand the $\gamma_j$'s in a more conventional setting.

Finally, we note that Killip--Simon \cite{KS2} have proven a sum rule and gem for half--line Schr\"odinger operators when $V \in L^2((0,\infty); dx)$.  It would be very interesting to find a large deviation proof of this result.  In particular, what is the analog of random matrix models for the study of Schr\"odinger operators.

\appendix
\section{}
\numberwithin{equation}{section}
We want to describe a proof of Theorem \ref{T4.2} that, because it is inductive, may be attractive to the reader.  We note that if one uses the GGT proof of the critical lemma below and uses explicit iteration instead of induction, our proof translates into a variant of the original proof of \cite{KN}.  We begin by writing $\bbU(n)$, the group of $n \times n$ unitary matrices, as a product of $\bbU(n-1)$ and $\bbC^n_1$, the set of unit vectors in $\bbC$.  Since topologically $\bbU(n)$ is not the product of $\bbU(n-1)$ and $\mathbb{S}^{2n-1}$, the unit sphere in $\bbR^{2n}$, this association cannot be continuous, but it will be measurable.  Critically, Haar measure on $\bbU(n)$ will be a product measure of Haar measure on $\bbU(n-1)$ and the rotation invariant measure on $\bbC^n_1$.

This idea goes back at least to Diaconis-Shahshahani \cite{DSh} who discussed the relation of Haar measure on $G$ to the natural product measure on $H \times G/H$ in general and illustrated this for the orthogonal group, $\mathbb{O}(n)$, using the Householder algorithm.  The discussion below is just the $\bbU(n)$--analog of their discussion of $\mathbb{O}(n)$. To be specific, let $G$ be a compact group and $H$ a closed subgroup of $G$.  Let $\pi:G \to G/H$ be the canonical projection.  Normalized Haar measure, $\mu_G$, induces a natural probability measure, $\mu_{G/H}$, on $G/H$ via
\begin{equation}\label{4.3a}
  \mu_{G/H}(A) = \mu_G(\pi^{-1}[A])
\end{equation}
and this measure is clearly invariant under the action of $G$ on $G/H$.

Let $\sigma: G/H \to G$ be a choice of representative from each coset, i.e.\ $\pi(\sigma(x)) = x$.  Then $\Sigma:G/H \times H \to G$, defined by $\Sigma(x,h) = \sigma(x)h$, is a bijection.  If one can choose $\sigma$ to be continuous, then $G$ will be homeomorphic to $G/H \times H$ under $\Sigma$ and often such a homeomorphism doesn't exist, e.g.\ if $G = \bbU(n)$ and $H=\bbU(n-1)$, so we should avoid the assumption that $\sigma$ is continuous.  It is probably true that in general one can make a measurable choice. Since we'll find an explicit such choice below for the case of interest we shall simply suppose that $\sigma$ is measurable.

\begin{proposition}  [Diaconis-Shahshahani \cite{DSh}] \lb{P4.2AA}  Suppose $\sigma$ is measurable.  Then under the bijection $\Sigma$ of $G/H \times H$ and $G$, the measure $\mu_{G/H} \otimes \mu_H$ goes to $\mu_G$.
\end{proposition}

\begin{proof} Let $U \in G, \, x \in G/H$.  Then $\pi(U\sigma(x)) = Ux$ so for some $W_{U,x} \in H$, we have that
\begin{equation}\label{4.3b}
U\sigma(x)=\sigma(Ux)W_{U,x}
\end{equation}
so $U\Sigma(x,W)=\Sigma(Ux,W_{U,x}W)$ which, given the invariance of $\mu_{G/H}$ under the action of $G$ and of $\mu_H$ under left multiplication by elements of $H$, implies the image of the product measure is invariant under multiplication by $U$ (by integrating first over $W$ and then $x$).
\end{proof}

Returning to $\bbU(n)$, fix a unit vector $e_1 \in \bbC^n_1$ (it may be helpful to think of $e_1$ as the first vector, $\delta_1=(1,0,\dots,0)$, of the canonical basis of $\bbC^n$).  The map $U \mapsto Ue_1$ is a surjection of $\bbU(n)$ to $\bbC^n_1$.  The inverse image of $e_1$ is those unitaries of the form $U = \bdone \oplus W$, under the direct sum decomposition $\bbC^n = [e_1] \oplus [e_1]^\perp$, where $W$ is an arbitrary unitary on $[e_1]^\perp$. Thus the set of $W's$ is isomorphic to $\bbU(n-1)$ and, if $e_1=\delta_1$, is canonically equal to it.  This shows that the quotient group $\bbU(n)/\bbU(n-1)$  of left cosets of $U(n-1)$ is exactly $\bbC^n_1$.

To realize $\bbU(n)$ as a product of $\bbU(n-1)$ and $\bbC^n_1$, we must pick, for each $f \in \bbC^n_1$, an element $\sigma(f) \in \bbU(n)$ so that $\sigma(f) e_1 = f$, i.e. $\sigma(f)$ is in the coset associated to $f$.  By the above noted fact about topological products, we cannot make this choice continuous in $f$, but one can make it measurable, indeed continuous on $\bbC^n_1 \setminus \left \{ \bbC\cdot e_1 \right \}$, as follows.  Suppose $f$ is not colinear with $e_1$.  Then $e_1,f$ span a two dimensional subspace $\calH_f$.  We can pick another vector $e_2 \in \calH_f$ orthonormal to $e_1$ specifying it uniquely by demanding that
\begin{equation}\label{4.4a}
  \kappa \equiv \jap{f,e_2} > 0
\end{equation}
We also define
\begin{equation}\label{4.4b}
  \beta \equiv \overbar{\jap{f,e_1}}
\end{equation}
so that $\beta \in \bbD$ and
\begin{equation}\label{4.4c}
  f = \bar{\beta}e_1+\kappa e_2
\end{equation}
Since $f$ is a unit vector
\begin{equation}\label{4.4d}
  \kappa = \sqrt{1-|\beta|^2}
\end{equation}

There is an obvious unitary map on $\calH_f$ that takes $e_1$ to $f$, namely reflection, $\Theta(\beta)$, in the line along $e_1+f$, which is clearly
\begin{equation}\label{4.4d1}
      \bdone -2\jap{e_1-f,\cdot}(e_1-f)/\norm{e_1-f}^2
\end{equation}
To find its matrix form in the $e_1,e_2$ basis, we note by \eqref{4.4c}, that its first column must be $\left(
                                                                                                         \begin{array}{c}
                                                                                                           \bar{\beta} \\
                                                                                                           \kappa \\
                                                                                                         \end{array}
                                                                                                       \right)$ since it takes $e_1$ to $f$.  Its second column is then determined by orthonormality and the desire to have determinant -1 (i.e.\ a reflection).  Thus
\begin{equation}\label{4.4e}
  \Theta(\beta) = \left(
                    \begin{array}{cc}
                      \bar{\beta} & \kappa \\
                      \kappa & -\beta \\
                    \end{array}
                  \right)
\end{equation}
We define the Householder reflection, $\sigma(f)$, on $\bbC^n$ to be
$\Theta(\beta) \oplus \bdone_{n-2}$ under
$\bbC^n = \calH_f \oplus \calH_f^\perp$
$($where $\bdone_k$ is the size $k$ identity matrix$)$.
$\sigma(f)$ is given by \eqref{4.4d1}, now as an operator on $\bbC^n$.  This formula makes it clear that $f \mapsto \sigma(f)$ is continuous on $\bbC^n \setminus \{\bbC\cdot e_1\}$ and discontinuous at the points of $\bbC\cdot e_1$.  We define $\sigma$ at $\lambda e_1$ to be $\lambda \bdone$. We thus have:

\begin{proposition} [Diaconis-Shahshahani \cite{DSh}] \lb{P4.2A}  Every unitary $U \in \bbU(n)$ can be uniquely written $\sigma(f) W$ where $f=Ue_1$ and $W$ is a unitary map on $[e_1]^\perp$. This maps $\bbU(n)$ Borel bijectively to $\bbC^n_1 \times \bbU(n-1)$.  Under this map, Haar measure on $\bbU(n)$ is just the product of the rotation invariant measure on $\bbC^n_1 \cong \mathbb{S}^{2n-1}$ and Haar measure on $\bbU(n-1)$.
\end{proposition}

The final assertion follows from Proposition \ref{P4.2AA}.

To link this to OPUC, assume that $e_1$ is cyclic for $U$ and let $d\mu$ be the spectral measure associated to $(U, e_1)$. By the spectral theorem, there is a unitary transformation $V: \bbC^n \to L^2(\partial\bbD, d\mu)$ such that $Ve_1=1$ (the constant function) and $(VUV^{-1}h)z=zh(z)$, so that $(VUe_1)(z)=z$. In terms of the orthogonal polynomials w.r.t.\ $\mu$,
this means that $Ve_1=\Phi_0=\frac{\Phi_0}{\|\Phi_0\|}=:\varphi_0$  and
$Ve_2=\frac{\Phi_1}{\| \Phi_1\|}=:\varphi_1$
(the ortho\emph{normal} polynomials). 
Szeg\H{o} recursion \eqref{1.2} for the normalized polynomials says (with $\rho_0=\sqrt{1-|\alpha_0|^2}$)
\begin{equation}\label{4.4f}
  z\varphi_0 = \rho_0 \varphi_1 + \bar{\alpha}_0 \varphi_0
\end{equation}
so comparing with \eqref{4.4c}, we see that $\beta = \alpha_0$ and $\kappa = \rho_0$ and
\begin{equation}\label{4.4g}
  \Theta(\alpha) = \left(
                    \begin{array}{cc}
                      \bar{\alpha} & \rho \\
                      \rho & -\alpha \\
                    \end{array}
                  \right)
\end{equation}

The key lemma is

\begin{lemma} \lb{L4.2B} Let $e_1$ and $e_2$ be two orthonormal vectors in $\bbC^n_1$.  Let $U'$ be a unitary on $[e_1]^\perp$ and $\Theta(\tilde{\alpha})$ given by \eqref{4.4g} in $e_1,e_2$ basis.  Let
\begin{equation}\label{4.4h}
  U = [\Theta(\tilde{\alpha})\oplus \bdone_{n-2}][\bdone_1 \oplus U']
\end{equation}
%
Then $e_1$ is cyclic for $U$ if and only if $e_2$ is cyclic for $U'$.  If they are cyclic, the Verblunsky coefficients, $\{\alpha_j\}_{j=0}^{n-1}$ for $(U,e_1)$ and  $\{\beta_j\}_{j=0}^{n-2}$ for $(U',e_2)$ are related by
\begin{equation}\label{4.4i}
  \alpha_0 = \tilde{\alpha}; \qquad \alpha_j = \beta_{j-1}, \, j=1,\dots,n-1
\end{equation}
\end{lemma}

\begin{remark}The proof will rely on matrix realizations of multiplication by $z$ on $L^2(\partial\bbD,d\mu)$, the subject of \cite[Chapter 4]{OPUC1}.  There are three such representations there.  First the GGT matrix, $\calG$, which uses the orthonormal basis of $\bbC^n$ (assuming $\mu$ is an $n$--point measure) obtained by applying Gram--Schmidt to $1,z,z^2,\dots$ (i.e.\ the orthonormal polynomials).  Second the CMV matrix, $\calC$, in the basis obtained by using Gram--Schmidt on $1,z,z^{-1},z^2,z^{-2},\dots$, and finally the alternate CMV matrix, $\wti{\calC}$, obtained using the basis obtained by applying Gram--Schmidt to $1,z^{-1},z,z^{-2},z^2,\dots$.  Below we'll give a proof exploiting $\calC$ and $\wti{\calC}$ and afterwards indicate a proof using $\calG$.
\end{remark}

\begin{proof} The cyclicity statement is a simple calculation.

We recall the $\calL\calM$ factorization of the CMV matrix \cite[Section 4.2]{OPUC1}.  Define $\Theta_{-1}$ to be the $1 \times 1$ identity matrix and, if $|\alpha| = 1$, $\Theta(\alpha)$ is the $1 \times 1$ matrix with element $\bar{\alpha}$.  If $\alpha \in \bbD$, $\Theta(\alpha)$ is given by \eqref{4.4g}.  If $n$ is even, we define as operators on $\bbC^n$:
\begin{equation}\label{4.4j}
  \calL = \Theta(\alpha_0) \oplus \Theta(\alpha_2) \oplus \dots \oplus \Theta(\alpha_{n-2});  \qquad \calM = \Theta_{-1} \oplus \Theta(\alpha_1) \oplus \dots \oplus \Theta(\alpha_{n-1})
\end{equation}
and if $n$ is odd
\begin{equation}\label{4.4k}
  \calL = \Theta(\alpha_0) \oplus \Theta(\alpha_2) \oplus \dots \oplus \Theta(\alpha_{n-1}); \qquad \calM = \Theta_{-1} \oplus \Theta(\alpha_1) \oplus \dots \oplus \Theta(\alpha_{n-2})
\end{equation}

If $\mu$ is an $n$--point measure, let $\{\chi_j\}_{j=0}^{n-1}$ be the orthonormal basis obtained using Gram--Schmidt on $1,z,z^{-1},\dots,z^k$ if $n=2k$ and $1,z,z^{-1},\dots,z^k,z^{-k}$ if $n=2k+1$.  Let $\{x_j\}_{j=0}^{n-1}$ be the same for $1,z^{-1},z\dots,z^{-k}$ if $n=2k$ or $1,z^{-1},z,\dots,z^{-k},z^k$ if $n=2k+1$.  One defines the CMV and alternate CMV matrices by
\begin{equation}\label{4.4l}
  \calC_{k\ell} = \jap{\chi_k,z\chi_\ell}, \qquad \wti{\calC}_{k\ell} = \jap{x_k,zx_\ell}
\end{equation}
Then
\begin{equation}\label{4.4m}
  \calL_{k\ell} = \jap{\chi_k,zx_\ell}, \qquad \calM_{k\ell} = \jap{x_k,\chi_\ell}
\end{equation}
so we have that
\begin{equation}\label{4.4n}
  \calC = \calL \calM; \qquad \wti{\calC} = \calM \calL
\end{equation}

Clearly
\begin{align}
  \calL(\alpha_0,\alpha_2,\dots) &= [\Theta(\alpha_0)\oplus \bdone_{n-2}][\bdone_2 \oplus \calL(\alpha_2,\dots)] \lb{4.4o} \\
                                 &= [\Theta(\alpha_0)\oplus \bdone_{n-2}][\bdone_1 \oplus \calM(\alpha_2,\dots)] \lb{4.4p}
\end{align}

Clearly,
\begin{equation}\label{4.4q}
  \calM(\alpha_1,\alpha_3,\dots) = \bdone_1 \oplus \calL(\alpha_1,\alpha_3,\dots)
\end{equation}
and this, combined with \eqref{4.4p}, implies the critical
\begin{equation}\label{4.4r}
  \calC(\alpha_0,\alpha_1,\dots) = [\Theta(\alpha_0) \oplus \bdone_{n-2}][\bdone_1 \oplus \wti{\calC}(\alpha_1,\alpha_2,\dots)]
\end{equation}

One consequence of \eqref{4.4r} is that the map $(\alpha_0,\alpha_1,\dots) \mapsto \calC(\alpha_0,\alpha_1,\dots)$ is injective, a critical part of a spectral theoretic proof of Verblunsky's Theorem (\cite[Theorem 4.2.8]{OPUC1} has a more awkward proof of this fact),
for $\alpha_0 = \overbar{\jap{\delta_1,\calC(\alpha_0,\alpha_1,\dots)\delta_1}}, \, \alpha_1 = \overbar{\jap{\delta_2,(\Theta(\alpha_0) \oplus \bdone_{n-2})^{-1} \calC(\alpha_0,\alpha_1,\dots) \delta_2}}, \,  \alpha_2 = \overbar{\jap{\delta_3,(\Theta(\alpha_0) \oplus \bdone_{n-2})^{-1} \calC(\alpha_0,\alpha_1,\dots)(\bdone_1 \oplus \Theta(\alpha_1) \oplus \bdone_{n-3})^{-1} \delta_3}}, \, \dots$.
Moreover, this inductive argument shows the initial basis is the CMV basis for the matrix.  A similar set of arguments work for the alternate CMV matrix and its basis

Another consequence of \eqref{4.4r} is the proof of \eqref{4.4i}.  For given $U,U',e_1,e_2$ as in the hypothesis, $e_1$ and $e_2$ are the first two elements in the CMV basis for $(U,e_1)$.  Extend $e_1,e_2$ to the full CMV basis for $(U,e_1)$ (call its elements $\{e_1,e_2,e_3,\ldots,e_n\}$).  In this basis, $U=\calC(\alpha_0,\alpha_1,\dots,\alpha_{n-1})$, so \eqref{4.4h} and \eqref{4.4r} imply that $U'= \bdone_1 \oplus \wti{\calC}(\alpha_0,\alpha_1,\dots,\alpha_{n-1})$.  By the remark at the end of the previous paragraph and \eqref{4.4r}, $\{e_2,e_3,\ldots,e_n \}$ is the alternate CMV basis for $\wti{\calC}$ and so, since $\wti{\calC}$ determines its $\alpha$'s, we conclude \eqref{4.4i}.
\end{proof}

\begin{remark} For GGT matrices, what Simon \cite{SimonCMV} calls the AGR factorization after Ammar, Gragg and Reichel \cite{AGR}, implies that
\begin{equation}\label{4.4s}
  \calG(\alpha_0,\alpha_1,\dots\alpha_{n-1}) = [\Theta({\alpha_0}) \oplus \bdone_{n-2}][\bdone_1 \oplus \calG(\alpha_1,\dots,\alpha_{n-1})]
\end{equation}
This can replace \eqref{4.4r} in an alternate proof of the Lemma and this alternate proof is related to the argument in \cite{KN}.
Moreover, Simon \cite{SimonCMV} has a version of the critical Lemma A.3 proven using the AGR factorization.

This also provides new insights into the difference between the GGT and CMV representations.  For $n$ even, define $\wti{\calL}_j, \, j=0,2,\dots, n$ as for $\calL$ with $\Theta_0,\dots,\Theta_{j-2},\Theta_{j+2},\dots,\Theta_{n-2}$ replaced by zero (only $\Theta_j$ remains in the direct sum) and similarly for $\wti{\calM}_j, \, j=-1,1,\dots,{n-1}$ (where $\Theta_{-1}, \Theta_{n-1}$ are $1 \times 1$ matrices.  Thus we have that
\begin{equation}\label{4.4t}
  \calL=\wti{\calL}_0 + \wti{\calL}_2 + \dots \wti{\calL}_{n-2}; \qquad \calM=\wti{\calM}_{-1}+\dots+\wti{\calM}_{n-1}
\end{equation}
Then the $\wti{\calL}$'s and $\wti{\calM}$'s are all Householder reflections and
\begin{equation}\label{4.4u}
  \calC = \wti{\calL}_0 \wti{\calL}_2 \wti{\calL}_{n-2} \dots \wti{\calM}_1 \wti{\calM}_3, \dots \wti{\calM}_{n-1}; \quad \calG = \wti{\calL}_0 \wti{\calM}_1 \wti{\calL}_2 \wti{\calM}_3 \dots\wti{\calM}_{n-1}
\end{equation}
with similar formulae if $n$ is odd.
\end{remark}

\begin{proof} [Proof of Theorem \ref{T4.2}]
Fix a vector $e_1 \in \bbC^n_1$. For $U$ picked according to Haar measure, $e_1$ is cyclic with probability one. Applying Proposition \ref{P4.2A}, $\wti{\bbP}_n(\alpha_0,\dots,\alpha_{n-1})$ is a product measure where $\alpha_0=\overline{\left(e_1,Ue_1 \right)}$ is distributed according to the distribution of $z_1$ if $\mathbf{z}$ is uniformly distributed on $\bbC^n_1$ which is exactly $d\kappa_{n-2}$ (since $\tfrac{n-1}{\pi}(1-|v_1|^2)^{n-2}$ is the size of the ``slice'' $\{(v_2,\dots,v_n) \in \bbC^n \, | \, |(v_2,\dots,v_n)|^2 = (1 - |v_1|)^2)\}$).

The other factor is what Haar measure for $\bbU(n-1)$ induces on the Verblunsky coefficients, $\beta_0,\dots,\beta_{n-2}$, for $\sigma(f)^{-1}U$.  By Lemma \ref{L4.2B}, the $\beta$'s are given by \eqref{4.4i}, so the result follows by induction.
\end{proof}


\end{document}